\documentclass[reqno, 12pt]{amsart}

\usepackage{amsfonts, amsthm, amsmath, amssymb}

\usepackage{hyperref}

\RequirePackage{mathrsfs} \let\mathcal\mathscr

\numberwithin{equation}{section}

\newtheorem{theorem}{Theorem}[section]
\newtheorem{lemma}[theorem]{Lemma}

\theoremstyle{definition}
\newtheorem*{ack}{Acknowledgements}

\renewcommand{\phi}{\varphi}
\renewcommand{\rho}{\varrho}
\newcommand{\1}{\mathbf{1}}

\newcommand{\sumstar}{\sideset{}{^*}\sum}

\newcommand{\PP}{\mathbb{P}}
\renewcommand{\AA}{\mathbb{A}}

\newcommand{\FF}{\mathbb{F}}
\newcommand{\ZZ}{\mathbb{Z}}

\newcommand{\NN}{\mathbb{Z}_{>0}}
\newcommand{\QQ}{\mathbb{Q}}
\newcommand{\RR}{\mathbb{R}}

\newcommand{\cA}{\mathcal{A}}

\newcommand{\cP}{\mathcal{P}}

\renewcommand{\leq}{\leqslant}

\renewcommand{\geq}{\geqslant}
\renewcommand{\ge}{\geqslant}
\renewcommand{\bar}{\overline}

\newcommand{\ma}{\mathbf}

\newcommand{\n}{\mathbf{n}}

\newcommand{\x}{\mathbf{x}}
\newcommand{\y}{\mathbf{y}}
\renewcommand{\c}{\mathbf{c}}

\newcommand{\ve}{\varepsilon}
\newcommand{\e}{\mathrm e}

\DeclareMathOperator{\Mod}{mod} 
\renewcommand{\bmod}[1]{\,(\Mod{#1})}

\renewcommand{\=}{\equiv}

\begin{document}

\title[The polynomial sieve]{The polynomial sieve and equal sums\\ of like polynomials}

\author{T.D.\ Browning}
\address{School of Mathematics\\
University of Bristol\\ Bristol\\ BS8 1TW\\ United Kingdom}
\email{t.d.browning@bristol.ac.uk}

\date{\today}

\thanks{2010  {\em Mathematics Subject Classification.} 11N35 (11N36, 11P05)}

\begin{abstract}
A new ``polynomial sieve'' 
is presented and used to 
show that almost all integers have at most one representation as a sum of two values of a given polynomial of   degree at least $3$.  
\end{abstract}	

\dedicatory{\centering{
Dedicated to \'Etienne Fouvry on his sixtieth birthday}}

\maketitle

\section{Introduction}

Suppose that we are given a set $\cA\subset \ZZ^m$. A primary goal 
in sieve theory is to estimate how many elements of $\cA$ have  components belonging to a particular sequence of  integers, such as squares, for example. 
Let  $w:\ZZ^m\rightarrow \RR_{\geq 0}$ be  a non-negative weight function 
 such that 
$$
\sum_{\n\in \ZZ^m}w(\n)<\infty.
$$
Let  $f(x;\y)\in \ZZ[x,\y]$ be a  polynomial, with $\y=(y_1,\dots,y_m)$, which 
we suppose takes the shape
$$
f(x;\y)=c_0(\y)x^d+\dots+c_d(\y),
$$
for polynomials $c_0,\dots,c_d\in \ZZ[\y]$ such that $c_0$ does not vanish identically.
In particular $f(x;\y)$ has degree $d$ with respect to $x$.

We seek an upper bound for the sum
$$
S(\cA)=\sum_{\substack{\n\in \cA\\ \mbox{\scriptsize{$f(x;\n)$ soluble}} }} w(\n),
$$
where for 
$\n\in \cA$  solubility of $f(x;\n)$ means that  there exists $x\in \ZZ$ such that $f(x;\n)=0$.
In order to prevent this condition being vacuous,
it is natural to restrict attention to $\n\in \cA$ for which 
$f(x;\n)$ does not vanish identically, 
Moreover, we will  introduce extra flexibility into our bound for $S(\cA)$ by allowing $w$ to be 
supported away from the zeros of a given auxiliary polynomial.
Our work is inspired by Heath-Brown's square sieve \cite{square}, which  corresponds to the special  case $m=1$ and 
$f(x;y)=x^2-y$.

\begin{theorem}\label{t:1}
Let $\cP$ be a set of primes, with $P=\#\cP$. Let
 $\alpha\in \NN$ and  let $g\in \ZZ[\y]$ be a non-zero polynomial.
For each $p\in \cP$ and $\n\in \ZZ^m$, let
$$
h(\n)=\gcd(c_0(\n),\ldots,c_d(\n))$$
and 
$$
\nu_p(\n)=\#\{x \bmod{p}: f(x;\n)\=0\bmod{p}\}.
$$
Suppose that $w(\n)=0$ if $g(\n)h(\n)=0$ or if
$|\n|\geq \exp(P)$.   
Then we have
$$
S(\cA)
\ll 
\frac{1}{P^2}\sum_{p,q\in \cP}\left|
\sum_{i,j\in \{0,1,2\}} c_{i,j}(\alpha)S_{i,j}(p,q)
\right|,
$$
with 
$$
S_{i,j}(p,q)=\sum_{\substack{\n\in \cA\\ \gcd(pq,g(\n)h(\n))=1}} w(\n) \nu_p(\n)^i\nu_q(\n)^j
$$
and 
$$
c_{i,j}(\alpha)=
\begin{cases}
(\alpha-d)^2, & \mbox{if $(i,j)=(0,0)$,}\\
\alpha+
(\alpha-1)d
-d^2, & \mbox{if $(i,j)= (1,0)$ or $(0,1)$,}\\
(1+d)^2, & \mbox{if $(i,j)=(1,1)$,}\\
-\alpha+d, & \mbox{if $(i,j)= (2,0)$ or $(0,2)$,}\\
-1-d, & \mbox{if $(i,j)= (2,1)$ or $(1,2)$,}\\
1, & \mbox{if $(i,j)=(2,2)$.}
\end{cases}
$$
\end{theorem}

This result will be established in \S \ref{s:sieve}.
The implied constant  is allowed to depend on the polynomials 
 $f\in \ZZ[x,\y]$  and $g\in\ZZ[\y]$.

The parameter $\alpha\geq 1$ in Theorem \ref{t:1} should be thought of as bounded absolutely in terms of $d$ and $m$. 
Our upper bound for  $S(\cA)$ leads us to study the sums 
$S_{i,j}(p,q)$ for suitable primes $p$ and $q$.
In favourable circumstances it will be possible to get an asymptotic formula for each of these sums, with appropriate main terms $M_{i,j}(p,q)$. The idea would then be  to choose $\alpha\geq 1$ in such a way that the sum 
$\sum_{i,j} c_{i,j}(\alpha) M_{i,j}(p,q)$ vanishes.

Theorem \ref{t:1} is a generalisation of the square sieve of Heath-Brown \cite{square}. To see this we take $m=1$, $f(x;y)=x^2-y$ and $g(y)=2y$ in our result. Then $d=2$, $h(n)=1$ and $\nu_p(n)=1+(\frac{n}{p})$ if $p>2$.
A direct calculation shows that 
\begin{align*}
\sum_{i,j\in \{0,1,2\}} c_{i,j}(\alpha)&\nu_p(n)^i\nu_q(n)^j\\
&=
(\alpha-1)^2+(\alpha-1)\left\{\left(\frac{n}{p}\right)+ \left(\frac{n}{q}\right)\right\} +\left(\frac{n}{pq}\right),
\end{align*}
if $\gcd(pq,2n)=1$.
We are  led to take $\alpha=1$ in Theorem \ref{t:1}. 
Then, if  $p=q$ is an odd prime in $\cP$,  we  deduce that 
$$
\sum_{i,j}c_{i,j}(1)S_{i,j}(p,q)\leq \sum_{n\in \cA}w(n).
$$ It therefore follows that 
\begin{align*}
S(\cA)\ll \frac{1}{P}\sum_{n\in \cA}w(n) + \frac{1}{P^2} 
\sum_{\substack{p\neq q\in \cP}}  \left|
\sum_{n\in \cA} w(n) \left(\frac{n}{pq}\right)
\right|,
\end{align*}
which recovers \cite[Thm.~1]{square} exactly.
In a similar fashion, by taking 
$m=1$, $f(x;y)=x^d-y$  and $g(y)=dy$, it is possible to deduce the power sieve of Munshi
\cite[Lemma~2.1]{munshi} from Theorem \ref{t:1}.

\medskip

We will illustrate  Theorem \ref{t:1} by investigating the numbers that can be represented as the sum of two values of a given polynomial.  
Let $f\in  \ZZ[x]$ be a polynomial of degree $d\geq 3$ with positive leading coefficient. Consider 
the arithmetic function
$$
r_f(n)=\#\{(y,z)\in \NN^2:  n=f(y)+f(z)\}.
$$
The 
average behaviour of $r_f(n)$ is easily understood with recourse to the geometry of numbers, with the outcome that 
there is a constant $c_f>0$ such that 
\begin{equation}\label{r(n)}
\sum_{n\leq N}r_f(n)\sim c_fN^{2/d}, \qquad 
(N\rightarrow \infty).
\end{equation}
The following result provides an estimate for its second moment.

\begin{theorem}\label{t:2}
We have 
$$
\sum_{n\leq N} r_f(n)^2\sim 2c_f N^{2/d}, \qquad 
(N\rightarrow \infty).
$$
There are asymptotically $\frac{1}{2}c_fN^{2/d}$   
integers $n\leq N$ for which $r_f(n)\neq 0$, and almost all
of these have essentially just one representation.
\end{theorem}

In fact this result may be further quantified in the following manner. 
For $B\geq 1$, let 
$E_f(B)$ denote the number 
of positive integers $y_1,y_2,y_3,y_4\leq B$ such that 
\begin{equation}\label{eq:Xs}
f(y_1)+f(y_2)=f(y_{3})+f(y_{4}),
\end{equation}
with  $\{y_1,y_2\}\neq \{y_3,y_4\}$. 
The sum in Theorem \ref{t:2}
counts solutions of \eqref{eq:Xs} in positive integers 
$y_1,\dots,y_4$ with 
$f(y_{1})+f(y_{2})\leq N$.
Any solution in which
$y_{3},y_{4}$ are not a permutation of $y_{1},y_{2}$ will
be counted by $E_{f}(B)$, for $B$ of order $N^{1/d}$.
Amongst the trivial solutions,  
there will be $O(N^{1/d})$ in which $y_{1}=y_{2}$, 
whence
\[\sum_{n\leq N}
r_f(n)^{2}=2\sum_{n\leq N}r_f(n)+O(N^{1/d}+E_f(cN^{1/d})),
\]
for an appropriate constant $c>0$.   
The first part of Theorem \ref{t:2} will therefore follow from \eqref{r(n)}, if we are able to show that $E_f(B)=o(B^2)$.
The second part is standard (see  the deduction of Theorem 2 from Theorem 1 in \cite{hooley-cubes}, for example, 
which deals with the  case $f(x)=x^3$).

Assuming that $d\geq 3$, we would like to show that  
there exists  $\delta>0$ such that 
\begin{equation}\label{eq:E}
E_f(B)=O_{f}(B^{2-\delta}),
\end{equation}
which clearly suffices for the first part of 
Theorem \ref{t:2}.  It is in the special case $f(x)=x^d$ that this
quantity  has received the most attention. Although there have been  subsequent refinements by many authors, it follows from work of Hooley \cite{hoo1, hoo2} that one can take any $\delta<1/3$ in \eqref{eq:E} when $f(x)=x^d$.
For general polynomials $f\in \ZZ[x]$ of degree $d\geq 3$, progress has not been so fluid. 
For $d=3$,  Wooley \cite{wooley} has shown that any $\delta<1/3$ is admissible in \eqref{eq:E}. For $d\geq 7$, previous work of the author \cite[Thm.~1]{me} shows that any 
$
\delta<5/6-2/\sqrt{7}=0.077\dots
$
is admissible.  
This was extended in joint work of the author with Heath-Brown \cite[Cor.~3]{smoothI}, where for $d\geq 5$ any $\delta<3/4-\sqrt{5}/3=0.004\dots$ is shown to be admissible
in \eqref{eq:E}.
It therefore remains to deal with the case $d=4$.

\begin{theorem}\label{t:3}
Let $\ve>0$ and let 
$f \in \ZZ[x]$ be a non-zero quartic polynomial.  Then we have 
$
E_f(B)
\ll_{\ve,f} B^{2-1/6+\ve}.
$
\end{theorem}

Our proof of Theorem \ref{t:3} will follow the strategy
of Hooley \cite{hoo1, hoo2}  for the case
 $f(x)=x^d$, except that we invoke Theorem \ref{t:1} rather than the generalised Selberg sieve adopted by Hooley. 
While this doesn't afford stronger results it does result in a more straightforward exposition. 
The lack of homogeneity that comes from  treating general polynomials $f(x)$ leads to several additional complications when  estimating the  emergent exponential sums. This ultimately leads to a weaker exponent
 in Theorem \ref{t:3}, compared with Hooley's exponent $5/3+\ve$ when $f(x)=x^4$.
However, in this special case, our argument can easily be modified to  recover this  exponent.

\begin{ack}
The author is indebted to Roger Heath-Brown for 
discussions at the research programme 
``Rational and integral points on higher-dimensional varieties'', at MSRI in 2006,
which led to Theorem \ref{t:1} taking shape. His contribution to the resulting paper is gratefully acknowledged. 
While working on this paper the  author was 
supported by  ERC grant \texttt{306457}.
\end{ack}

\section{Proof of Theorem \ref{t:1}}\label{s:sieve}

Our argument is a generalisation of the proof of \cite[Thm.~1]{square}. 
It will be convenient to write $\nu_p=\nu_p(\n)$ in what follows, for each $p\in \cP$. 
We begin by considering the expression
$$
\Sigma=\sum_{\n\in \cA} w(\n)\left(	 
\sum_{\substack{p\in \cP\\ p\nmid g(\n)h(\n)}}  \left\{ \alpha+(\nu_p-1)(d-\nu_p)\right\}
\right)^2.
$$
Each $\n$ is clearly counted with non-negative weight.  Suppose now that $\n\in \cA$ is such that $f(x;\n)$ is soluble and 
$g(\n)h(\n)\neq 0$.
 Then 
$
1\leq \nu_p\leq d
$ 
for every $p\in \cP$ such that $p\nmid h(\n).$
Hence it follows that  
$$
\alpha+(\nu_p-1)(d-\nu_p)\geq \alpha	\geq 1
$$
in the summand, 
whence
$$
\sum_{\substack{p\in \cP\\ p\nmid g(\n)h(\n)}} \left\{ \alpha+(\nu_p-1)(d-\nu_p)\right\}\geq 
\sum_{\substack{p\in \cP\\ p\nmid g(\n)h(\n)}}
1\geq P - \sum_{p\mid g(\n)h(\n)}1,
$$
if $f(x;\n)$ is soluble.
But 
$$
\sum_{p\mid N}1\ll \frac{\log N}{\log\log 3N},
$$ 
for any $N\in \NN$.
It follows from our assumptions on the support of $w$ that 
 $\Sigma\gg P^2 S(\cA)$, with an implied constant that depends 
on the polynomials  $f\in \ZZ[x,\y]$  and $g\in\ZZ[\y]$.

A companion estimate for $\Sigma$ is achieved by  expanding the square, giving the upper bound
$$
\sum_{p,q\in \cP} \left|\sum_{\substack{\n\in \cA\\ \gcd(pq,h(\n)g(\n))=1}} 
\hspace{-0.6cm}
w(\n)\left\{
\alpha+(\nu_p-1)(d-\nu_p)\right\}
\left\{
\alpha+(\nu_q-1)(d-\nu_q)\right\}\right|.
$$
Multiplying out the summand and then comparing this  with our lower bound for $\Sigma$, we easily arrive at the statement of Theorem \ref{t:1}.

\section{Proof of Theorem \ref{t:3} --- preliminaries}

Throughout the proof of Theorem \ref{t:3} we will allow all implied constants to depend in any way upon $f$. Any further dependencies will be indicated explicitly by appropriate subscripts.
Suppose that  $f(x)=a_0x^4+\dots+a_4$ for $a_0,\dots,a_4\in \ZZ$ and $a_0>0$.   Note that 
$$
4^4 a_0^3f(x)=(4a_0x+a_1)^4+b_2(4a_0x+a_1)^2+b_3(4a_0x+a_1)+b_4,
$$
for $b_2,b_3,b_4\in \QQ$ depending on $a_0,\dots,a_4$.
After a possible change of variables it therefore  suffices to establish Theorem \ref{t:3} for the monic polynomial
$$
f(x)=x^4+ax^2+bx,
$$
for given $a,b\in \ZZ$. Furthermore, we
may henceforth assume that $(a,b)\neq (0,0)$, since otherwise Theorem \ref{t:3} is a consequence of  work of Greaves \cite{greaves}, which shows that Theorem \ref{t:3} holds with exponent $2-\frac{1}{4}+\ve$.

In any given point $\y=(y_1,\dots,y_4)$ counted by $E_f(B)$ we may assume without loss of generality that $\max_i y_i=y_1$ and $y_3\geq y_4$.   It follows that 
$y_1>y_3\geq y_4>y_2\geq 0$.
Our starting point will be the factorisation properties of the equivalent equation
\begin{align*}
f(y_1)-f(y_3)&=f(y_4)-f(y_2).
\end{align*}
Through the substitutions
\begin{align*}
u_1=y_1-y_3, &\quad v_1=y_4-y_2, \\
u_2=y_1+y_3, &\quad v_2=y_4+y_2,
\end{align*}
this  equation transforms into
\begin{equation}\label{eq:tom'}
u_1(u_2^3+u_1^2u_2+2au_2+2b)=
v_1(v_2^3+v_1^2v_2+2av_2+2b).
\end{equation}
We observe that 
$u_1,u_2,v_1,v_2$ are  positive integers of size at most $2B$.
Moreover, 
$u_2\neq v_2$ since otherwise we would 
have $y_4+y_2-y_3=y_1>y_3$, from which it would follow that 
$2y_3>y_4+y_2>2y_3$, which is impossible.
We may further assume that $u_1\neq v_1$, since the remaining contribution is $O(1)$. Indeed, if $u_1=v_1$ then 
 our equation becomes
$$
u_2^2+u_2v_2+v_2^2+u_1^2=-2a,
$$
since $u_2\neq v_2$. 
This has $O(1)$ solutions in positive integers $u_1,u_2,v_2$.

We will analyse the Diophantine equation \eqref{eq:tom'} by drawing out common factors between $u_1$ and $v_1$. Given the extra symmetry inherent when $b=0$, we will also need to draw out common factors between $u_2$ and $v_1$. Let us write
$$
h_1=\gcd(u_1,v_1), \quad h_2=\gcd(u_2,v_1/h_1).
$$
We then make the change of variables 
\begin{align*}
(r,s)=(u_1/h_1,u_2/h_2), \quad (\rho,\sigma)=(v_1/(h_1h_2),v_2),
\end{align*}
with $\gcd(r,h_2\rho)=\gcd(s,\rho)=1$. 
Moreover, since $u_i\neq v_i$ for $i=1,2$ we may assume that 
$r\neq h_2\rho$ and $h_2s\neq \sigma$ in any solution.  
These variables satisfy the new equation
$$
r(h_2^3s^3+h_1^2h_2r^2s+2ah_2s+2b)=
h_2\rho(\sigma^3+h_1^2h_2^2\rho^2\sigma+2a\sigma+2b).
$$
In particular, $h_2\mid 2b$ since $\gcd(r,h_2)=1$.  Let us write 
$2b=h_2c$, for $c\in \ZZ$. Then we have 
\begin{equation}\label{eq:tom}
r(h_2^2s^3+h_1^2r^2s+2as+c)=
\rho(\sigma^3+h_1^2h_2^2\rho^2\sigma+2a\sigma+h_2c).
\end{equation}
Here $h_1h_2\leq 2B$ and  $r,s,\rho,\sigma$ 
are positive integers 
satisfying
$$
\gcd(r,h_2\rho)=\gcd(s,\rho)=1 \quad \mbox{and} \quad
(r-h_2\rho)(h_2s- \sigma)\neq 0,
$$
together with the inequalities
$$
r\leq \frac{2B}{h_1}, \quad 
s\leq \frac{2B}{h_2}, \quad 
\rho\leq \frac{2B}{h_1h_2}, \quad 
\sigma\leq 2B.
$$

Define the 
number
\begin{equation}\label{eq:def-A}
A=h_1^2h_2^2 \rho^2+2a.
\end{equation}
When $|A|$ is small or  $\max\{h_1,h_2\}$ is large, 
 we will use 
work of  Bombieri and Pila \cite{BP} to estimate the corresponding contribution.
In the alternative case, we will ultimately apply Theorem \ref{t:1}.
Let $C\geq 1$ and let $1\leq H\leq 2B$. Let
$N_{1}(B,C;H)$ 
(resp.~ $N_{2}(B,C;H)$) denote the total contribution to $E_f(B)$ from solutions with  
$|A|> C$ and 
$\max\{h_1,h_2\}\leq H$ (resp.~ 
$|A|\leq C$ or $\max\{h_1,h_2\}>H$). 
Then our work so far implies that 
$$
E_f(B)\leq N_1(B,C;H)+N_2(B,C;H) +O(1).
$$
The treatment of the second term is relatively  straightforward.

\begin{lemma}\label{lem:big-h}
Let $\ve>0$. Then 
$$
N_{2}(B,C;H) \ll_\ve C^{1+\ve}B^{4/3+\ve}+H^{-1}B^{7/3+\ve}.
$$
\end{lemma}

\begin{proof}
One way to estimate the number of solutions to \eqref{eq:tom} is to first fix some of the variables, viewing  the resulting equation as something of smaller dimension. 
Let $C_{\ma{h},\rho,r}\subset \AA_\QQ^2$ denote the affine cubic curve which arises when $\ma{h}=(h_1,h_2)$ and $\rho,r$ are fixed. 
Let us put 
$A'=
h_1^2r^2+2a
$, for ease of notation. 
Then we claim that $C_{\ma{h},\rho,r}$ is absolutely irreducible unless 
\begin{equation}\label{eq:reduced}
c=0 \quad \mbox{and} \quad h_2^2\rho^2A^3=r^2A'^3,
\end{equation}
with $A\neq 0$.
To prove this we suppose that 
$C_{\ma{h},\rho,r}$ is not absolutely irreducible. Then
it must contain a line defined over $\bar \QQ.$  We may assume that this line is given parametrically by $(s,\sigma)=(t, \alpha t+\beta)$ for $\alpha, \beta\in \bar \QQ$. 
Making this substitution into \eqref{eq:tom} and equating coefficients of $t$, we deduce that 
$$
\beta=c=0 \quad \mbox{and} \quad rh_2^2=\rho \alpha^3 \quad \mbox{and} \quad 
rA'=\rho A\alpha,
$$
with $A\neq 0$, 
since $h_1h_2\rho(r\pm h_2\rho)\neq 0$.
Eliminating $\alpha$ easily leads to the claim.

Suppose that  $\ma{h}, \rho, r$ do not satisfy \eqref{eq:reduced}. 
It then follows from 
a result of Bombieri and Pila \cite{BP} that 
\begin{equation}\label{eq:BP-app}
\#\{s,\sigma\leq 2B: (s,\sigma)\in C_{\ma{h},\rho,r}(\ZZ)\}
=O_\ve(B^{1/3+\ve}),
\end{equation}
for any $\ve>0$. The implied constant is independent of  $\ma{h},\rho,r$ and depends only on $\ve$.
Alternatively, if  $\ma{h}, \rho, r$ do satisfy \eqref{eq:reduced} then we have the trivial bound 
$O(B/h_2)$ for the number of points in
$C_{\ma{h},\rho,r}(\ZZ)$, which arises from noting that there are at most $3$ choices of $\sigma$ associated to a given choice of  $s$.

We may now handle the contribution from $|A|\leq C$, in which case $h_1h_2\rho\ll C$. 
There are $O_\ve(C^{1+\ve})$  choices for $h_1,h_2,\rho$ satisfying this bound, by the trivial estimate for the divisor function. 
When 
$\ma{h}, \rho, r$ do not satisfy \eqref{eq:reduced} we will apply \eqref{eq:BP-app}. This case therefore gives an overall contribution
$O_\ve (C^{1+\ve}B^{4/3+\ve} )$. Alternatively, when
$\ma{h}, \rho, r$ do satisfy \eqref{eq:reduced} there are at most  $8$  choices for $r$ when $h_1,h_2,\rho$ are fixed. 
This case therefore makes the smaller  overall contribution 
$O_\ve (C^{1+\ve}B)$.

Next, let us consider the contribution from $h_1>H$.
We fix a choice of $\ma{h}, r$ and $\rho$ in \eqref{eq:tom}.
When \eqref{eq:reduced} fails 
we may apply \eqref{eq:BP-app}. This 
leads to the contribution 
$$
\ll_\ve B^{1/3+\ve} \sum_{h_1>H} \sum_{h_2\leq 2B} \frac{B^2}{h_1^2h_2}\ll_\ve H^{-1}B^{7/3+\ve}\log B.
$$
Taking $\log B=O_\ve(B^\ve)$ and redefining the choice of $\ve>0$, this is satisfactory for the lemma.
Alternatively, when \eqref{eq:reduced} is satisfied
we apply the bound $O(B/h_2)$ for the number of $s,\sigma$. 
But then $\rho,r$ are restricted by the equation
$h_2^2\rho^2A^3=r^2A'^3$, which once reduced modulo $h_1^2$ implies that 
$$
8a^3(h_2^2\rho^2-r^2)\equiv 0\bmod{h_1^2}.
$$
We must have $a\neq 0$ since  $c=0$ and we are assuming that $(a,b)\neq (0,0)$.
Let $q=h_1^2/\gcd(h_1^2,8a^3)$.
Then this congruence becomes 
$h_2^2\rho^2\equiv r^2\bmod{q}.$
Write $q'=q/\gcd(q,2)$.
Since $\gcd(r, h_2\rho)=1$ we deduce that 
$r\equiv h_2\rho \bmod{q'}$ or 
$r\equiv -h_2\rho \bmod{q'}$.
In particular we must have $q'\ll B/h_1$, since $0\neq r\pm h_2\rho\ll B/h_1$.
In either case, given $\ma{h}, \rho$ we see that the number of $r$ that can possibly contribute is 
$$
\ll\frac{B}{h_1q'}\ll \frac{B}{h_1^3},
$$
and to each of these is associated at most $8$ choices for $\rho$.
This case therefore leads to the 
overall contribution 
$$
\ll B \sum_{h_1>H} \sum_{h_2\leq 2B} \frac{B^2}{h_1^3h_2}\ll H^{-2}B^{2}\log B,
$$
which is satisfactory.

It remains to consider the contribution from $h_2>H$. This is handled in a completely analagous fashion, 
by first fixing a choice of 
$\ma{h}, s$ and $\rho$ and
considering the 
affine cubic curve $D_{\ma{h},\rho,s}\subset \AA_\QQ^2$.
In this case, on writing 
$A''=h_2^2s^2+2a$, 
one finds that  $D_{\ma{h},\rho,s}$ is absolutely irreducible unless 
$$
c=0 \quad \mbox{and} \quad h_1^2\rho^2A^3=s^2A''^3,
$$
with $A\neq 0$.
When $D_{\ma{h},\rho,s}$ is absolutely irreducible one 
applies the analogue of \eqref{eq:BP-app}. When it fails to be 
 absolutely irreducible one applies the trivial bound $O(B/h_1)$ for the number of points 
in $D_{\ma{h},\rho,s}(\ZZ)$
The remainder of the argument runs just as before. 
This concludes the proof of the lemma.
\end{proof}

The estimation of $N_{1}(B,C;H)$ is much more awkward. The remainder of this paper is dedicated to proving 
the following result. 

\begin{lemma}\label{lem:small-h}
Let $\ve>0$ and assume that $C\gg 1$. Then we have 
$$
N_{1}(B,C;H) \ll_\ve   H^{1/2}B^{3/2+\ve}
+B^{2-1/6+\ve}.
$$
\end{lemma}

Once this result is combined with Lemma \ref{lem:big-h}, we see that the choices $C\ll 1$ and $H=B^{1/2}$ are sufficient to establish  Theorem~\ref{t:3}.

We now begin the proof of Lemma \ref{lem:small-h}.
Our plan  will be to fix  choices of $h_1,h_2$ and  $\rho$, and then to count the number of $r,s,\sigma$ that contribute to $N_1(B,C;H)$.  
Define the cubic polynomials
\begin{align}
\label{eq:F}
F(u,v)&=
h_2^2v^3+(h_1^2u^2+2a)v+c,\\
\label{eq:F'}
G(u,v)&=
v^3+(h_1^2h_2^2u^2+2a)v+h_2c,
\end{align}
where we recall that $2b=h_2c$.
Then \eqref{eq:tom} can be written 
\begin{equation}\label{eq:stare}
rF(r,s)=\rho G(\rho,\sigma).
\end{equation}

Recall the definition \eqref{eq:def-A} of $A$. We are proceeding under the assumption that 
$|A|>C\geq 1$. In particular $A\neq 0$. 
Part of our work will lead us to consider the  homogeneous quartic polynomial
\begin{equation}\label{eq:def-K}
\begin{split}
K(Z,X,Y,W)=~&W^4h_1h_2\rho G(\rho,Z/W)\\
&-2\{W^4f(X/W)-W^4f(Y/W)\},
\end{split}
\end{equation}
where $G$ is given by \eqref{eq:F'}. 
The condition on $C$ in Lemma \ref{lem:small-h} comes from  the following 
result.

\begin{lemma}\label{lem:non}
Assume that $C\gg 1$. Then $K$ is non-singular.
\end{lemma}

\begin{proof}
We recall that $h_1h_2\rho\neq 0$ and $|A|>C$.
Returning to \eqref{eq:def-K} we see, by taking partial derivatives, that any singular point on the projective surface $K=0$ must satisfy 
$$
W(3Z^2+AW^2)=0
$$
and 
$$
4X^3+aXW^2+bW^3=0, \quad
4Y^3+2aYW^2+bW^3=0,
$$
in addition to $\partial K/\partial W=0$. A short calculation shows that
the latter constraint is equivalent to the equation
$$
h_1h_2\rho F(Z,W)=2\left\{2a(X^2-Y^2)W+3b(X-Y)W^2\right\},
$$
where $F(Z,W)=Z^3+3AZW^2+4h_2cW^3$. There can be no singular points with $W=0$. Hence it follows that there are at most $18$ singular points on $K=0$, and these all take the shape
$[\xi, \eta,\eta',1]$, where 
$$
\xi=\pm \sqrt{-A/3}
$$
and $\eta,\eta'$ are roots of the cubic equation $4t^3+2at+b=0$.
In particular, it follows that 
$
h_1h_2\rho F(\xi,1)\ll 1,
$
which is impossible provided that $C$ is taken to be sufficiently large in our lower bound
$|A|>C$. Hence there are no singular points, which thereby establishes  the lemma.
\end{proof}

We proceed to indicate  how the polynomial sieve will be brought to bear on the proof of Lemma \ref{lem:small-h}. The structure of our argument is modelled on that of Hooley \cite{hoo2}, corresponding to the special case $f(x)=x^4$.
We shall assume that $C\gg 1$ for the remainder of the proof, so that Lemma \ref{lem:non} applies and $K$ is non-singular.
Since $\gcd(r,\rho)=1$, 
  it follows 
  from \eqref{eq:stare}
  that $\rho\mid F(r,s)$ in any solution to be counted.
We therefore have
\begin{equation}\label{eq:N1-1}
N_1(B,C;H)\leq \sum_{h_1,h_2\leq H}
\sum_{\substack{\rho\leq 2B/(h_1h_2)\\ |A|>C}}
N_1(B;H;\ma{h},\rho),
\end{equation}
where $A$ is given by \eqref{eq:def-A} and 
$N_1(B;H;\ma{h},\rho)$ is equal to 
$$
\sum_{\substack{r\leq 2B/h_1,~s\leq 2B/h_2\\
\gcd(rs,\rho)=1\\
F(r,s)\equiv 0\bmod{\rho}
}}
\times\begin{cases}
1, &\mbox{if $\exists \sigma\in \ZZ$ s.t. $
\rho G(\rho,\sigma)
=rF(r,s)$,}\\
0, & \mbox{otherwise}.
\end{cases}
$$
This is now in a   form suitable for an application of Theorem \ref{t:1}.

To be precise, we take 
$$\cA=
\left\{
(r,s)\in \ZZ^2\cap (0,2B/h_1]\times (0,2B/h_2]:
\begin{array}{l}
\gcd(rs,\rho)=1\\
F(r,s)\equiv 0\bmod{\rho}
\end{array}
  \right\}
 $$ 
and $w$ to be the indicator function for this set. We take  
$$
f(x;r,s)=\rho G(\rho,x)-rF(r,s)
$$
and $g(r,s)=1$.
Recalling \eqref{eq:F'} we have   $d=3$ and 
$h(r,s)\mid \rho$ in
Theorem \ref{t:1}. In particular $f(x;r,s)$ never vanishes identically, for any $(r,s)\in \cA$.
Let 
$\alpha\geq 1$ and $Q\geq 1$ be parameters. 
Let $D_K$ be the discriminant of the  quartic form $K$ in 
\eqref{eq:def-K}. Then $D_K$ is a non-zero integer 
since $K$ is non-singular. We let 
\begin{equation}\label{eq:P}
\cP=
\{\mbox{primes $p\leq Q$}: \mbox{
$p\nmid 6 h_1h_2\rho AD_K$}\}.
\end{equation}
In particular $K$ remains non-singular modulo any prime $p\in \cP$.
For any $p\in \cP$ and $(r,s)\in \cA$, we put
\begin{equation}\label{eq:nu}
\nu_p(r,s)=\#\{x\bmod{p}: \rho G(\rho,x)\equiv rF(r,s) \bmod{p}\}.
\end{equation}
We will always assume that $Q$ satisfies $B^{1/100}\leq Q\leq B$. In particular
$$
\#\cP=\pi(Q)
-\#\{p\leq Q: 
\mbox{$p\mid 6 h_1h_2\rho AD_K$}\} \sim \frac{Q}{\log Q},
$$   
by the prime number theorem.
It now follows from Theorem \ref{t:1} that 
\begin{equation}\label{eq:N1-2}
N_1(B;H;\ma{h},\rho)\ll \frac{\log^2 Q}{Q^2}\sum_{p,q\in \cP}\left|
\sum_{i,j\in \{0,1,2\}} c_{i,j}(\alpha)S_{i,j}
\right|,
\end{equation}
with 
$$
S_{i,j}=\sum_{\substack{(r,s)\in \cA}}  \nu_p(r,s)^i\nu_q(r,s)^j
$$
and 
\begin{equation}\label{eq:c-3}
c_{i,j}(\alpha)=
\begin{cases}
(\alpha-3)^2, & \mbox{if $(i,j)=(0,0)$,}\\
4(\alpha-3), & \mbox{if $(i,j)= (1,0)$ or $(0,1)$,}\\
16, & \mbox{if $(i,j)=(1,1)$,}\\
3-\alpha, & \mbox{if $(i,j)= (2,0)$ or $(0,2)$,}\\
-4, & \mbox{if $(i,j)= (2,1)$ or $(1,2)$,}\\
1, & \mbox{if $(i,j)=(2,2)$.}
\end{cases}
\end{equation}
We will ultimately be led to take $\alpha=1$ in \S \ref{s:conclusion}.

To analyse $S_{i,j}$ we will break the sum into congruence classes modulo $pq\rho$.
Let $Y\geq 1$ and let $N\in \ZZ$ with $|N|\leq pq\rho/2$. Then we have 
\begin{equation}\label{eq:Gamma}
\Gamma(Y,N)=\sum_{y\leq Y} \e_{pq\rho} (-Ny)\ll \min\left\{Y, \frac{pq\rho}{|N|}\right\}.
\end{equation}
Given $r\in \ZZ$  the orthogonality of characters yields
\begin{align*}
\#\{x\leq 2B/h_1: ~& x\equiv r \bmod{pq\rho}\}\\
&=
\frac{1}{pq\rho} \sum_{m\bmod{pq\rho}}  \e_{pq\rho} (mr)
\sum_{x\leq 2B/h_1}
\e_{pq\rho} (-mx)\\
&=
\frac{1}{pq\rho} \sum_{-pq\rho/2<m\leq pq\rho/2}  \e_{pq\rho} (mr)
\Gamma\left(\frac{2B}{h_1},m\right),
\end{align*}
and similarly for 
$\#\{y\leq 2B/h_2: y\equiv s \bmod{pq\rho}\}$.
Hence 
\begin{equation}\label{eq:Sij}
S_{i,j}=
\frac{1}{(pq\rho)^2} 
\hspace{-0.2cm}
\sum_{-pq\rho/2<m,n\leq pq\rho/2}  
\hspace{-0.4cm}
\Gamma\left(\frac{2B}{h_1},m\right)
\Gamma\left(\frac{2B}{h_2},n\right)\Psi_{i,j}(m,n),
\end{equation}
with 
$$
\Psi_{i,j}(m,n)
=
\sum_{\substack{(r,s) \bmod{pq\rho}\\ \gcd(\rho,rs)=1\\
F(r,s)\equiv 0\bmod{\rho}
}}  \nu_p(r,s)^i\nu_q(r,s)^j
 \e_{pq\rho} (mr+ns).
$$
It therefore remains to understand the exponential sums
$\Psi_{i,j}(m,n)$.
For typical values of $m,n$ we want  to show that there is enough cancellation in the sum to make its modulus rather small.   Recall from the definition \eqref{eq:P} of $\cP$ that $\gcd(pq, \rho)=1$. Using this, we are able to establish the following factorisation property. 

\begin{lemma}\label{lem:fact}
Suppose that $p\neq q$ and choose $p',q',\bar{pq},\bar{\rho}\in \ZZ$ such that 
$pq\bar{pq}+\rho\bar{\rho}=1$
and 
$pp'+qq'=1$.
Then we have 
$$
\Psi_{i,j}(m,n)=\Sigma_i(p;\bar{\rho}q'm,\bar{\rho}q'n)\Sigma_j(q;\bar{\rho}p'm,\bar{\rho}p'n)\Phi(\rho;\bar{pq}m,\bar{pq}n), $$
where
\begin{align}
\label{eq:sig-i}
\Sigma_t(p;M,N)&=
\sum_{\substack{(r,s) \bmod{p}}}
 \nu_p(r,s)^t
 \e_{p} (Mr+Ns),\\
\label{eq:sig}
\Phi(\rho;M,N)&=
 \sum_{\substack{(r,s) \bmod{\rho}\\ \gcd(\rho,rs)=1\\
F(r,s)\equiv 0\bmod{\rho}
}}  \e_{\rho} (Mr+Ns).
\end{align}
Suppose that $p=q$ and choose $\bar{p},\bar{\rho}\in \ZZ$ such that 
$p\bar{p}+\rho\bar{\rho}=1$.
Then we have 
$$
\Psi_{i,j}(m,n)=
\begin{cases}
p^2 
\Sigma_{i+j}(p;\bar{\rho}m',\bar{\rho}n')\Phi(\rho;\bar{p}m',\bar{p}n'), & 
\mbox{if $(m,n)=p(m',n')$,}\\
0, & \mbox{otherwise}.
\end{cases}
$$
\end{lemma}

\begin{proof}
The proof of this result is standard. The first part is obtained by making the substitution 
$$
r=(r_0qq'+r_1pp')\rho\bar{\rho}+r_2 pq\bar{pq}, \quad 
s=(s_0qq'+s_1pp')\rho\bar{\rho}+s_2 pq\bar{pq}, 
$$
for $r_0,s_0 \bmod{p}$, $r_1,s_1 \bmod{q}$ and 
$r_2,s_2 \bmod{\rho}$, with $r_2s_2$ coprime to $\rho$. 
For the second part we make the substitution
$$
r=r_1\rho\bar{\rho}+r_2 (p\bar{p})^2, \quad 
s=s_1\rho\bar{\rho}+s_2 (p\bar{p})^2, 
$$
for $r_1,s_1 \bmod{p^2}$ and 
$r_2,s_2 \bmod{\rho}$, 
with $r_2s_2$ coprime to $\rho$.
This leads to the expression
$$
\Psi_{i,j}(m,n)=\Phi(\rho;\bar{p}^2m,\bar{p}^2n)
\sum_{\substack{(r_1,s_1) \bmod{p^2}}}
\nu_p(r_1,s_1)^{i+j}
\e_{p^2} (\bar{\rho}\{mr_1+ns_1\}),
$$
in the notation of the lemma. Writing $r_1=r_1'+pr_1''$ for $r_1',r_1''\bmod{p}$, and similarly for $s_1$, the second factor becomes
$$
\sum_{\substack{(r_1',s_1') \bmod{p}}}
\hspace{-0.2cm}
\nu_p(r_1',s_1')^{i+j}
\e_{p^2} (\bar{\rho}\{mr_1'+ns_1'\})
\hspace{-0.3cm}
\sum_{\substack{(r_1'',s_1'') \bmod{p}}}
\e_{p} (\bar{\rho}\{mr_1''+ns_1''\}).
$$
But the inner sum is zero unless $p\mid \gcd(m,n)$, in which case it is $p^2$.
This completes the proof of the lemma.
\end{proof}

We have reduced our task to a detailed analysis of the exponential sums 
$\Sigma_t(p;M,N)$ and $\Phi(\rho;M,N)$, for $0\leq t\leq 4$ and given $M,N\in \ZZ$.
This will be the object of the following two sections. 
The trivial bound for 
$\Sigma_t(p;M,N)$ is  $O(p^2)$. Likewise, in the special case that $\rho$ is a prime, the trivial bound for $\Phi(\rho;M,N)$ is $O(\rho)$. In our work we will show that for generic choices of $M,N$ these sums actually 
satisfy square-root cancellation.
We will do so using 
 the Weil bound for curves and the Deligne bound for surfaces, combined with an elementary treatment of $\Phi(\rho;M,N)$ when $\rho$ is a higher prime power.
We prepare the ground by framing some basic tools that will be common to both. 
Given any sum over residue classes, we will use $\sum^*$ to mean a sum in which all the variables of summation are coprime to the modulus.

Our primary means of estimating the exponential sums for prime modulus will be the  ``method of moments'' developed by Hooley \cite{hooleysum}, as summarised in the following result. 

\begin{lemma}\label{genhoo}
Let $F$ and $G_1,\dots,G_k$ be polynomials over $\ZZ$, of degree at most $d$, and let
$$
S=\sum_{\substack{\x\in\mathbb{F}_p^n\\ G_1(\x)=\dots=G_k(\x)=0}}\e_p(F(\x)), 
$$ 
for any prime $p$. For each $j\ge 1$ and $\tau\in \FF_{p^j}$, write
$$
N_j(\tau)=\#\{\x\in\mathbb{F}_{p^j}^n:\,G_1(\x)=\dots= G_k(\x)=0,\,F(\x)=\tau\},
$$
and suppose that  there exists
 $N_j\in \RR$ such that
\begin{equation}\label{hoo1}
\sum_{\tau\in\mathbb{F}_{p^j}}|N_j(\tau)-N_j|^2\ll_{d,k,n} p^{\kappa j},
\end{equation}
where $\kappa\in \ZZ$ is independent of $j$.  Then
$
S\ll_{d,k,n}p^{\kappa/2}.
$
\end{lemma}

Let  $q$ be a prime power. We 
shall need to be able to count $\FF_q$-points on certain varieties. 
In dimension $2$ we will call upon the work of  Deligne \cite{deligne} and in dimension $1$ we will use work of Weil \cite{weil}. The facts that we need are summarised in the following two results. 

\begin{lemma}\label{lem:deligne}
Let $W\subset \PP_{\FF_q}^n$ be a non-singular complete intersection of dimension $2$ and degree  $d$. 
Then 
$$
\#\{\x\in \FF_q^n: [\x]\in W\}=q^{3}+O_{d,n}(q^{2}).
$$
\end{lemma}

\begin{lemma}\label{lem:weil}
Let $V\subset \AA_{\FF_q}^n$ be an absolutely irreducible curve of degree $d$. 
Then 
$$
\#V(\FF_q)=q+O_{d,n}(q^{1/2}).
$$
\end{lemma}

\section{The exponential sums $\Sigma_t(p;M,N)$}

In this section we examine the exponential sum 
$\Sigma_t=\Sigma_t(p;M,N)$ in 
\eqref{eq:sig-i} for a prime  $p\nmid 6h_1h_2\rho A D_K$,
where $A$ is given by \eqref{eq:def-A} and $D_K$ is the non-zero discriminant of the quartic form \eqref{eq:def-K}.
For $i=1,2$, we let $\bar h_i\in \ZZ$ be such that 
$h_i\bar h_i \equiv 1 \bmod{p}$.
Recall the definitions \eqref{eq:F}, \eqref{eq:F'} of the cubic polynomials
$F$ and $G$. Reversing the changes of variables leading to these, one easily checks that 
\begin{equation}\label{eq:change}
h_1h_2uF(u,v)=2\left\{ f\left(\frac{h_1u+h_2v}{2}\right)-  f\left(\frac{-h_1u+h_2v}{2}\right) \right\},
\end{equation}
where $f(x)=x^4+ax^2+bx$. 
We will argue according to the value of $t$.

When $t=0$ we trivially have 
\begin{equation}\label{eq:sig-0}
\Sigma_0=
\begin{cases}
p^2, &\mbox{if $p\mid \gcd(M,N)$,}\\
0, & \mbox{otherwise}.
\end{cases}
\end{equation}
Next, when $t=1$, we open up the function $\nu_p(r,s)$ given by \eqref{eq:nu} to conclude that
$$
\Sigma_1
=
\sum_{\substack{(r,s,x)\in  \FF_p^3\\
\rho G(\rho,x)= rF(r,s) }} \e_p(Mr+Ns).
$$
We will show that 
\begin{equation}\label{eq:sig1'}
\Sigma_1=O\left(p\gcd(p,M,N)\right).
\end{equation}
On  carrying out the non-singular change of variables implicit in \eqref{eq:change},
 we obtain
\begin{equation}\label{eq:sig1}
\begin{split}
\Sigma_1
&=S_p( M\bar h_1+N\bar h_2, -M\bar h_1+ N\bar h_2),
\end{split}
\end{equation}
where  for $\c=(c_1,c_2)\in \ZZ^2$ we set
$$
S_p(\c)=\sum_{\substack{(x,y,z)\in \FF_p^3\\
h_1h_2\rho G(\rho,z)=2\{f(x)-f(y)\}
}}\e_p(c_1x +c_2y).
$$
Inserting the second part of the following result into \eqref{eq:sig1} establishes \eqref{eq:sig1'}.

\begin{lemma}\label{lem:S_p}
We have
$$
S_p(0,0)=p^2+O(p) \quad \mbox{and} \quad 
S_p(\c)=O(p\gcd(p,c_1,c_2)).
$$
\end{lemma}

\begin{proof}
We begin by establishing the first part of the lemma.
We convert the problem into one involving projective varieties via the identity
$$
S_p(0,0)=\frac{1}{p-1}
\#\{(z,x,y,w)\in \FF_p^4: 
K(z,x,y,w)=0, ~w\neq 0\},
$$
where $K$ is given by \eqref{eq:def-K}.
Combining Lemmas \ref{lem:deligne} and \ref{lem:non}, we see that 
$$
\#\{(z,x,y,w)\in \FF_p^4: 
K(z,x,y,w)=0\} = p^{3}+O(p^2),
$$
since $K$ is non-singular over $\FF_p$.
But $K(Z,X,Y,0)=Y^4-X^4$.
Hence
$$
\#\{(z,x,y)\in \FF_p^3: 
K(z,x,y,0)=0\} =O(p^2).
$$
Putting these two estimates together gives
 $S_p(0,0)=p^2+O(p)$.

Turning to $S_p(\c)$ for general $\c\in \ZZ^2$, we may assume that $p\gg 1$ and 
$p\nmid (c_1,c_2)$, else the bound follows from the first part of the lemma.
Let $j\geq 1$ and put $q=p^j$. Since 
$K$ is non-singular over $\FF_q$, it follows that  $K(z,x,y,1)$ must be absolutely irreducible over $\FF_q$.
We apply Lemma~\ref{genhoo} with $k=1$ and $n=3$. 
 Then, for $\tau\in \FF_q$,  we must consider
$$
N_j(\tau)=\#\{(x,y,z)\in \FF_q^3:
K(z,x,y,1)=0, ~c_1x +c_2y=\tau
\}.
$$
We reduce $c_1,c_2$ modulo $p$ and view them as elements of $\FF_q$. Without loss of generality we assume  that $c_1\neq 0$,  eliminating $x$ to get
$$
N_j(\tau)=\#\{(y,z)\in \FF_q^2:
K(z,-c_1^{-1}c_2y+c_1^{-1}\tau,y,1)=0
\}.
$$
It follow from Hilbert's irreducibility theorem that there are $O(1)$ values of $\tau\in \FF_q$ for which 
$K(z,-c_1^{-1}c_2y+c_1^{-1}\tau,y,1)$ fails to be absolutely irreducible. 
For these we employ the trivial bound $N_j(\tau)=O(q)$.
For the remaining values of $\tau$, 
 Lemma \ref{lem:weil} yields
$
N_j(\tau)=q+O(q^{1/2}),
$
uniformly in $\tau$.
Taking $N_j=q$ in Lemma \ref{genhoo} therefore permits the choice 
$\kappa=2$ in \eqref{hoo1}, which completes the proof of the lemma.
\end{proof}

We now turn to the case $t=2$.
Define the quadratic polynomial
$$
H(Z_1,Z_2,W)=Z_1^2+Z_1Z_2+Z_2^2+AW^2,
$$
where $A$ is given by \eqref{eq:def-A}. This quadratic form is non-singular modulo $p$, since $p\nmid A$ for any $p\in \cP$.  Next,  observe that 
\begin{equation}\label{eq:H}
\rho G(\rho,x_1)-\rho G(\rho,x_2)=\rho(x_1-x_2)H(x_1,x_2,1).
\end{equation}
Opening up $\nu_p(r,s)^2$ in $\Sigma_2$ gives
\begin{equation}\label{eq:bishop}
\Sigma_2=\Sigma_1+
\sum_{\substack{(r,s,x_1,x_2)\in  \FF_p^4\\
x_1-x_2\neq 0\\
H(x_1,x_2,1)=0\\
\rho G(\rho,x_1)=rF(r,s)}}
 \e_{p} (Mr+Ns),
\end{equation}
where $\Sigma_1$ is the contribution from  $x_1-x_2=0$.
We will show that 
\begin{equation}\label{eq:sig2'}
\Sigma_2=O\left(p\gcd(p,M,N)\right).
\end{equation}
We may remove the condition $x_1-x_2\neq 0$ 
 in the second sum of \eqref{eq:bishop} at the expense of an additional error term $O(p)$.
Hence, on making the change of variables implicit in \eqref{eq:change}, we obtain 
\begin{equation}\label{eq:sig2}
\Sigma_2=\Sigma_1+T_p(M\bar h_1+N\bar h_2,-M\bar h_1+N\bar h_2)+O(p),
\end{equation}
where
for $\c=(c_1,c_2)\in \ZZ^2$ we set
$$
T_p(\c)=
\sum_{\substack{(x,y,z_1,z_2)\in \FF_p^4\\
H(z_1,z_2,1)=
K(z_1,x,y,1)=0}}
 \e_{p} (c_1x+c_2y),
$$
where $K$ is given by \eqref{eq:def-K}.
The estimate   \eqref{eq:sig2'} will follow on
combining  Lemma \ref{lem:S_p} with the second part of the  following result in \eqref{eq:sig2}.

\begin{lemma}\label{lem:T_p}
We have
$$
T_p(0,0)=p^2+O(p) \quad \mbox{and}
\quad
T_p(\c)=O(p\gcd(p,c_1,c_2)).
$$
\end{lemma}

\begin{proof}
The proof of this result is similar to our argument in Lemma~\ref{lem:S_p}.  
Let $q=p^j$ for $j\geq 1$.  We may clearly assume that $p\gg 1$, since the estimates are trivial otherwise.
We begin  with the first estimate,  converting the problem into one involving projective varieties by noting that $T_p(0,0)$ is equal to
$$
\frac{1}{p-1}
\#\left\{(x,y,z_1,z_2,w)\in \FF_p^5: 
\begin{array}{l}
H(z_1,z_2,w)=K(z_1,x,y,w)=0\\ 
w\neq 0
\end{array}
\right\},
$$
The desired conclusion will follow from Lemma \ref{lem:deligne}, provided we can show that $H=K=0$ defines a non-singular surface in $\PP_{\FF_q}^4$.  

Delaying this for the moment, we move to an analysis of 
 $T_p(\c)$ for general $\c\in \ZZ^2$, with
$p\nmid (c_1,c_2)$.  
Still under the assumption that the projective variety  $H=K=0$ is non-singular, 
it follows  that $H(z_1,z_2,1)=K(z_1,x,y,1)=0$ defines an absolutely irreducible affine variety over $\FF_q$.
We apply Lemma~\ref{genhoo} with $k=2$ and $n=4$. 
On assuming without loss of generality  that $c_1\neq 0$, 
  we must consider
$$
N_j(\tau)=\#\left\{(y,z_1,z_2)\in \FF_q^3:
\begin{array}{l}
H(z_1,z_2,1)=0\\
K(z_1,-c_1^{-1}c_2y+c_1^{-1}\tau,y,1)=0
\end{array}{}
\right\},
$$
for $\tau\in \FF_q$.
By Hilbert's irreducibility theorem there are $O(1)$ values of $\tau$ for which 
the equations in $N_j(\tau)$ fail to define an absolutely irreducible curve in $\AA_{\FF_q}^3$. 
For these we take $N_j(\tau)=O(q)$.
For the remaining $\tau$, Lemma \ref{lem:weil} yields
$
N_j(\tau)=q+O(q^{1/2}),
$
uniformly in $\tau$.
Taking $N_j=q$ in Lemma \ref{genhoo} therefore permits the choice 
$\kappa=2$ in \eqref{hoo1}, which leads to the claimed bound for $T_p(\c).$

It remains to show that the equations  $H=K=0$ produce a non-singular variety   in $\PP_{\FF_q}^4$.  
For this we consider the existence of a non-zero point $(Z_1,Z_2,X,Y,W)$ such that 
$$
H=K=0, \quad \lambda\nabla H=\mu \nabla K,
$$
with  $(\lambda,\nu)\neq (0,0).$  
We have already remarked that $H$ and $K$ are  non-singular over $\FF_q$.
Hence we must have $\lambda\mu\neq 0$ in any such solution. Moreover, $W\neq 0$ in any such solution, since for $W=0$ the equation for $K$ implies that $X=Y=0$ and the remaining constraints force $Z_1=Z_2=0$. Next we observe that 
$
\partial H/\partial Z_2=0,
$
in any solution.  Likewise, on replacing $K(Z_1,X,Y,W)$ by $K(Z_2,X,Y,W)$, we may adjoin to this the equation
$\partial H/{\partial Z_1}=0.
$
Finally, since $H=0$ and $W\neq 0$, an application of Euler's identity implies that 
$
{\partial H}/{\partial W}=0,
$
which is impossible since $H$ is non-singular. Hence there are no singular points, as claimed.
\end{proof}

Suppose now that $t\geq 3$  and
write $\x=(x_1,\dots,x_t)$. 
 Opening up
 $\nu_p(r,s)^t$ in $\Sigma_t$ gives
$$
\Sigma_t=
\sum_{\substack{(r,s,\x)\in  \FF_p^{t+2}\\
i\neq j \Rightarrow (x_i-x_j)H(x_i,x_j,1)=0\\
\rho G(\rho,x_1)=rF(r,s)}}
 \e_{p} (Mr+Ns),
$$
via \eqref{eq:H}. 
Let $\sigma(\x)$ denote the number of distinct elements in the set $\{x_1,\dots,x_t\}$. Clearly $1\leq \sigma(\x)\leq t$.
The contribution from those $(r,s,\x)$ for which $\sigma(\x)=1$ is 
 $\Sigma_1$. This event can only arise in one way.
 The contribution 
from  those $(r,s,\x)$ for which $\sigma(\x)=2$ is 
$\Sigma_2-\Sigma_1$, by \eqref{eq:bishop}.
There are $c_t$ ways in which this can arise, for an appropriate constant $c_t$ depending on $t$. 
Next, consider the contribution from 
$(r,s,\x)$ for which $\sigma(\x)=3$.
Suppose, for example, that $\x=(x,y,z,x,\dots,x)$ with 
$
(x-y)(x-z)(y-z)\neq 0.
$
In this case $x,y,z$ will satisfy 
$$
0=H(x,y,1)=H(x,z,1),
$$
whence in fact $x+y+z=0$.    In view of \eqref{eq:bishop}, the contribution from this case is therefore found to be 
$$
\sum_{\substack{(r,s,x,y)\in \FF_p^4\\
(x-y)(2x+y)(x+2y)\neq 0\\
H(x,y,1)=0\\
\rho G(\rho,x)=rF(r,s)}}
 \e_{p} (Mr+Ns)=\Sigma_2-\Sigma_1+O(p).
$$
This situation  arises in $d_t$ ways, say.
Finally, our argument shows that there can be no contribution from $(r,s,\x)$ for which 
$\sigma(\x)\geq 4$.
It follows that 
\begin{equation}\label{eq:vicar}
\Sigma_t=\left\{1-c_t-d_t\right\}
\Sigma_1+
\left\{c_t+d_t\right\}
\Sigma_2+O_t(p),
\end{equation}
for $t\geq 3$.
We are now ready to record the following  result, which summarises our investigation in this section.

\begin{lemma}\label{lem:sigma-final}
We have 
$
\Sigma_t(p;M,N)=O_t\left(p\gcd(p,M,N)\right)
$
for $t\geq 0$, and 
$$
\Sigma_t(p;0,0)=
\max\{1,t\}
p^2+O(p)
$$
for $0\leq t\leq 2$.
\end{lemma}
\begin{proof}
The first part follows from 
\eqref{eq:sig-0}, \eqref{eq:sig1'}, \eqref{eq:sig2'} and \eqref{eq:vicar}. 
Turning to the second part, with   $M=N=0$, the case $t=0$ follows directly from 
\eqref{eq:sig-0} and  the case  $t=1$ follows from \eqref{eq:sig1} and  Lemma \ref{lem:S_p}. Finally, the case $t=2$ follows from \eqref{eq:sig2} and Lemmas~\ref{lem:S_p} and \ref{lem:T_p}. 
\end{proof}

\section{The exponential sum $\Phi(\rho;M,N)$}

Recall the definition \eqref{eq:sig} of the exponential sum $\Phi(\rho;M,N)$ for $\rho\in \NN$ and $M,N\in \ZZ$.
It will be convenient to define
\begin{equation}\label{eq:Delta}
\Delta(M,N)=h_2^2M^2+h_1^2N^2.
\end{equation}
Suppose that $\rho_1,\rho_2$ are coprime integers and let $\bar{\rho}_1,\bar{\rho}_2\in \ZZ$ be such that $\rho_1\bar{\rho}_1+\rho_2\bar{\rho}_2=1$. 
Then arguing as in the proof of Lemma	\ref{lem:fact} it is easy to see that 
\begin{equation}\label{eq:melt}
\Phi(\rho_1\rho_2;M,N)=
\Phi(\rho_1;\bar\rho_2M,\bar\rho_2N)
\Phi(\rho_2;\bar\rho_1M,\bar\rho_1N).
\end{equation}
This renders it sufficient to study $$\Phi(p^k)=\Phi(p^k;M,N)
=
 \sumstar_{\substack{(r,s) \bmod{p^k}\\ 
F(r,s)\equiv 0\bmod{p^k}
}}  \e_{p^k} (Mr+Ns)
$$
for a given prime power $p^k$, where $F$ is given by \eqref{eq:F}.

We begin by examining the case $k=1$. Suppose that $p>2$ and that  $\ell,m\in \ZZ$, with $p\nmid \ell$. 
We will use the familiar formula 
\begin{equation}
\label{eq:gauss}
\sum_{x \bmod{p}}\e_{p}(\ell x^2+mx)=
\varepsilon_p\sqrt{p} \left(\frac{\ell}{p}\right) \e_{p}(-\overline{4\ell}m^2),
\end{equation}
for the Gauss sum, where 
$
\varepsilon_p=1$ (resp.\ $\ve_p=i$)
if 
$p\equiv 1 \bmod{4}$ (resp.\ 
if $p\equiv 3 \bmod{4}$).
We may now establish the following result.

\begin{lemma}\label{lem:k=1}
We have $\Phi(p)\ll p^{1/2}\gcd(p,\Delta(M,N))^{1/2}$.
\end{lemma}

\begin{proof}
Recall that $(a,b)\neq (0,0)$. In view of 
the  bound $|\Phi(p)|\leq p^2$, we may assume that $p\nmid 2\gcd(a,c)$. Next we observe that
$$
|\Phi(p)|\leq \sumstar_{r\bmod{p}}
\hspace{-0.2cm}
\#\{
s\bmod{p}: h_2^2s^3+(h_1^2r^2+2a)s+c\=0\bmod{p}\}\leq 
3p,
$$ 
for $p\nmid 2\gcd(a,c)$.
When $p\nmid h_2$ this follows since there are 
 at most $3$ solutions of  the congruence $s^3+As+B\=0 \bmod{p}$,
 for given $A,B\in \ZZ$. 
 When $p\mid h_2$, but $p\nmid (h_1^2r^2+2a)$, there is a unique choice of $s$ for given  $r$, which is satisfactory. Finally, when   $p\mid h_2$ and  $p\mid h_1^2r^2+2a$ we must have $p\nmid h_1$, 
 since $p\nmid 2\gcd(a,c)$.  Then 
 there at most $p$ choices for $s$ but only at most $2$
for $r$, which is also satisfactory. 
 
 We may  assume that $p\nmid 2\Delta(M,N)\gcd(a,c)$ for the remainder of the proof.
Suppose that $p\mid h_1.$  In particular 
$\Delta(M,N)\=h_2^2M^2 \bmod{p}$ and so $p\nmid h_2M$.
We have 
\begin{align*}
\Phi(p)
&=
 \sumstar_{\substack{(r,s) \bmod{p}\\
h_2^2s^3+2as+c\equiv 0\bmod{p}
}}  \e_{p} (Mr+Ns)\\
&= 
c_p(M)
 \sumstar_{\substack{s\bmod{p}\\ 
h_2^2s^3+2as+c\equiv 0\bmod{p}
}}  \e_{p} (Ns),
\end{align*}
where $c_p(M)$ is the Ramanujan sum.
Our argument in the preceding paragraph shows that $|\Phi(p)|\leq 3\gcd(p,M)=3$ when $p\mid h_1$, which  is satisfactory for the lemma. 

Suppose next that $p\mid \gcd(c,M)$.
Then it follows that  $p\nmid 2ah_1N$, since $p\nmid 2\Delta(M,N)\gcd(a,c)$. Replacing $h_1r$ by $r$ and using additive characters to detect the congruence, 
we have 
\begin{align*}
\Phi(p)
&=
 \sumstar_{\substack{(r,s) \bmod{p}\\
r^2\=-h_2^2s^2-2a\bmod{p}
}}  \e_{p} (Ns)\\
&=
\frac{1}{p}
\sum_{\ell\bmod{p}}
\hspace{0.1cm}
 \sumstar_{\substack{(r,s) \bmod{p}}}
 \e_p\left(\ell(
 r^2+h_2^2s^2+2a)
+Ns\right).
\end{align*}
The contribution to the sum from $ \ell\=0\bmod{p}$ is easily seen to be $O(1)$.  Moreover, 
we may assume that $p\nmid h_2$ in the remaining sum, else we get $\Phi(p)=O(1)$ overall. 
Replacing $h_2s$ by $s$, we get
\begin{align*}
\Phi(p)
&=
\frac{1}{p}
\hspace{0.1cm}
\sumstar_{\ell\bmod{p}}
\hspace{0.1cm}
 \sumstar_{\substack{(r,s) \bmod{p}}}
 \e_p\left(\ell(
 r^2+s^2+2a)
+\bar{h_2}Ns\right) +O(1)\\
&=
\frac{1}{p}
\hspace{0.1cm}
\sumstar_{\ell\bmod{p}}
 \e_p(2a\ell)
 \sumstar_{\substack{r \bmod{p}}}
 \e_p(\ell r^2)
 \sumstar_{\substack{s \bmod{p}}}
 \e_p(\ell s^2
+\bar{h_2}Ns) +O(1).
\end{align*}
Applying \eqref{eq:gauss}, we see that 
$$
\sumstar_{\ell\bmod{p}}
\e_p(\ell r^2)=
\varepsilon_p\sqrt{p} \left(\frac{\ell}{p}\right) -1
$$
and 
$$
\sumstar_{\substack{s \bmod{p}}}
\e_p(\ell s^2
+\bar{h_2}Ns)
=\varepsilon_p\sqrt{p} \left(\frac{\ell}{p}\right) \e_{p}(-\overline{4\ell h_2^2}N^2)-1.
$$
Hence  
\begin{align*}
\Phi(p)
&=
\ve_p^2
\hspace{0.1cm}
\sumstar_{\ell\bmod{p}}
 \e_p(2a\ell)\e_{p}(-\overline{4\ell h_2^2}N^2)
+O(p^{1/2}).
\end{align*}
But this is $O(p^{1/2})$
by the Weil bound for the Kloosterman sum.
This is satisfactory and so we can henceforth assume that 
$p\nmid \gcd(c,M)$ and 
$p\nmid 2h_1\Delta(M,N)\gcd(a,c)$.

We have 
\begin{align*}
\Phi(p)
&=
 \sum_{\substack{(r,s)\in \FF_p^2 \\
F(r,s)=0
}}  \e_{p} (Mr+Ns) +O(1),
\end{align*}
with $F$ given by \eqref{eq:F}.
We will use Lemma \ref{genhoo} to estimate the sum, with $k=1$ and $n=2$.
Let $j\geq 1$ and put $q=p^j$. Then for $\tau\in \FF_q$ we must consider
$$
N_j(\tau)=\#\{ (r,s)\in \FF_q^2: F(r,s)=0, ~Mr+Ns=\tau \},
$$
where we reduce $M$ and $N$ modulo $p$ and view them as elements of $\FF_q$.
If $M\neq 0$ then, on recalling our expression \eqref{eq:F} for $F$,  we eliminate $r$ to get
$$
N_j(\tau)=\#\{ s\in \FF_q:   g(s)=0 \},
$$
where $g$ is a  polynomial of degree $3$ with non-zero leading  coefficient $h_2^2+M^{-2}h_1^2N^2$.
Hence  $N_j(\tau)\leq 3$ in this case. Suppose next that $M=0$. 
In particular $c\neq 0$ and  $N\neq 0$. We may eliminate $s$ to get
$$
N_j(\tau)=\#\{ r\in \FF_q:   h_1^2N^{-1}\tau r^2+h_2^2N^{-3}\tau^3+2aN^{-1}\tau+c=0 \}.
$$
Hence $N_j(\tau)\leq 2$ if $\tau\neq 0$ and  $N_j(0)=0$.
Taking $N_j=0$ in Lemma~\ref{genhoo} therefore allows us to take $\kappa=1$ in \eqref{hoo1}, whence
$\Phi(p)\ll p^{1/2}$, as required to conclude the proof of the lemma.
\end{proof}

It remains to consider the general case $k\geq 2$.
It will be useful to collect together some basic estimates for the number of solutions to various polynomial congruences. Let $\nu\geq 1$ and let $A,B,C,D\in \ZZ$.
Beginning with quadratic congruences, we observe that 
$$
\#\{x\bmod{p^\nu}: x^2+D \equiv0 \bmod{p^\nu}\}\leq 2 p^{\min\{v_p(D),\nu\}/2}.
$$
Let $\xi\geq 0$ and assume that $2\xi\leq \nu$.
It follows from this that 
\begin{equation}\label{eq:quad}
\begin{split}
\#\{x\bmod{p^\nu}: p^{2\xi} x^2&+D \equiv0 \bmod{p^\nu}\}\\
&\begin{cases}
\leq 2 p^{\xi+\min\{v_p(D),\nu\}/2}, &\mbox{if $2\xi\leq v_p(D)$,}\\
=0, &\mbox{otherwise.}
\end{cases}
\end{split}
\end{equation}
Next, if $p\nmid \gcd(A,B,C)$, we have 
\begin{equation}\label{eq:hux}
\#\{x \bmod{p^\nu}: Ax^3+Bx+C\=0\bmod{p^\nu}\}\leq 3 p^{v_p(\delta(A,B,C))/2},
\end{equation}
where $\delta(A,B,C)=-(4AB^3+27A^2C^2)$ is the underlying discriminant.
This is established by Huxley \cite{hux}, for example. 
We are now ready to establish the following result.

\begin{lemma}\label{lem:ab-not0}
Suppose that $k\geq 2$.  Then 
$
\Phi(p^k)\ll 
p^{k} \gcd(p^{[k/2]},h_1).
$
\end{lemma}

\begin{proof}
Define the integer  $\Delta=2^3a^3+3^3b^2$.   Our argument will differ according to whether or not $\Delta$ vanishes. 
 Throughout our argument we put
 $\xi_i=v_p(h_i)$, for $i=1,2$, with $h_i=p^{\xi_i}h_i'$.

Our starting point is an analysis of the 
quantity 
$$
M(\nu)=\#\{ s\bmod{p^\nu}: p\nmid s, ~ g(s)\=0\bmod{p^\nu}\},
$$
for $\nu\geq 1$, where
$g(s)=p^{2\xi_2}h_2'^2s^3+2as+c$. 
We will show that 
\begin{equation}\label{eq:M}
M(\nu)\ll 
\begin{cases}
1, & \mbox{if $\Delta\neq 0$,}\\
p^{[\nu/2]}, & \mbox{if $\Delta= 0$.}
\end{cases}
\end{equation}
Suppose first that $\Delta=0$. 
Then we must have 
$(a,b)=(-6t^2,8t^3),$
for some  non-zero integer $t$. 
In particular $h_2=O(1)$ and we observe that 
\begin{align*}
M(\nu)&\leq 
\#\{s\bmod{p^\nu}: 
(h_2s+4t)(h_2s-2t)^2\=0 \bmod{p^\nu}\}\\
&\ll p^{[\nu/2]},
\end{align*}
as required. 

Turning to the case $\Delta\neq 0$, we suppose that  $b=0$. Then  $a\neq 0$ 
and we now have 
\begin{align*}
M(\nu)
&\leq \#\{ s\bmod{p^\nu}: p^{2\xi_2}h_2'^2s^2+2a
\=0\bmod{p^\nu}\}.
\end{align*}
If $2\xi_2\leq \nu$, 
it now follows from \eqref{eq:quad} that 
$
M(\nu)\leq 2p^{v_p(2a)} \ll 1.
$
If $2\xi_2> \nu$, then we trivially  have $M(\nu)\ll 1$ since then $\nu\leq v_p(2a)$.
We  now suppose that $b\neq 0$. In particular $h_2=p^{\xi_2}h_2'=O(1)$.
Write 
$$
\gamma=\min\{2\xi_2,v_p(2a),v_p(c)\},
$$ 
so that also $p^\gamma=O(1)$.
We may assume   that $\gamma< \nu$, since otherwise 
 $M(\nu)\ll 1$ 
is trivial. 
Let us write $2a=p^{\gamma}a'$ and $c=p^{\gamma}c'$, 
so that 
$$
M(\nu)\leq p^{\gamma}\#\{ s\bmod{p^{\nu-\gamma}}: 
p^{2\xi_2-\gamma}h_2'^2s^3+a's+c'
\=0\bmod{p^{\nu-\gamma}}\}.
$$
Since the  cubic polynomial now has content coprime to $p$, we 
may  apply \eqref{eq:hux} to deduce that 
$
M(\nu)\leq 3 p^{\gamma/2+\xi_2+v_p(d)/2}$,
where $d$ is the integer
\begin{align*}
d&=2^2a'^3+3^3p^{2\xi_2-\gamma}h_2'^2c'^2\\
&=p^{-3\gamma} \{2^5a^3+3^3h_2^2c^2\}\\
&=4p^{-3\gamma} \Delta.
\end{align*}
Since $\Delta\neq 0$ it follows that  $p^{v_p(d)}\ll 1$ and so
$M(\nu)\ll 1$, as required to complete the proof of  \eqref{eq:M}.

We are now ready  to establish the bound for $\Phi(p^k)$ in the lemma,  observing that 
$|\Phi(p^k)|$ is at most 
\begin{align*}
\#\{r,s\bmod{p^k}: p\nmid rs, ~h_2^2 s^3+(h_1^2r^2+2a)s+c\=0 \bmod{p^k}\}.
\end{align*}
If $2\xi_1\geq k$ then it follows from \eqref{eq:M} that 
$\Phi(p^k)\ll p^{k+[k/2]}$, which is satisfactory.  Alternatively, if $2\xi_1<k$, we deduce from \eqref{eq:quad} and \eqref{eq:M} that 
\begin{align*}
|\Phi(p^k)|
&\leq 2
\sumstar_{\substack{s\bmod{p^k}\\ 2\xi_1\leq v_p(g(s))}}
p^{\xi_1+\min\{v_p(g(s)), k\}/2 }\\
&\leq 2p^{\xi_1}
\sum_{2\xi_1\leq j<k} p^{k-j/2} M(j)
+2p^{\xi_1+k/2}M(k)\\
&\ll p^{k+\xi_1},
\end{align*}
which is also satisfactory.
This completes the proof of the lemma.
\end{proof}

We now collect together our work so far to deduce a general bound 
for the exponential sum $\Phi(\rho;M,N)$ using the multiplicativity property \eqref{eq:melt}.
Given $\rho\in \NN$, we will write 
$\rho=uvw^2$, where
\begin{equation}\label{eq:reading}
u=\prod_{\substack{p\| \rho}} p, \quad 
v=\prod_{\substack{p^j\| \rho\\ j\geq 2, ~2\nmid j}} p.
\end{equation}
Clearly $v$ divides $w$.
Drawing together 
 Lemmas \ref{lem:k=1} and \ref{lem:ab-not0} we easily arrive at the following result.

\begin{lemma}\label{lem:rho}
There exists an absolute constant $A>0$ such that 
$$
\Phi(\rho;M,N)\leq A^{\omega(\rho)} u^{1/2}vw^2\gcd(u,\Delta(M,N))^{1/2}\gcd(w,h_1),
$$
where $\Delta$ is given by \eqref{eq:Delta}.
\end{lemma}

\section{Conclusion}\label{s:conclusion}

It is now time to bring everything together in \eqref{eq:Sij}. 
Using the basic estimate
$[\theta]=\theta+O(1)$, the contribution to $S_{i,j}$ from the term $m=n=0$ is seen to be
\begin{align*}
\frac{\Psi_{i,j}(0,0)}{(pq\rho)^2} 
\left[\frac{2B}{h_1}\right] \left[\frac{2B}{h_2}\right]
&=M_{i,j}+O\left(\frac{\Psi_{i,j}(0,0)B}{
\min\{h_1,h_2	\}(pq\rho)^2
}\right),
\end{align*}
with 
\begin{equation}\label{eq:Mij}
M_{i,j}
=\frac{4\Psi_{i,j}(0,0)B^2}{h_1h_2(pq\rho)^2}.
\end{equation}
Hence it follows from \eqref{eq:Gamma} that 
\begin{equation}\label{eq:Sij'}
\begin{split}
S_{i,j}=~& M_{i,j}+
O\left(\frac{\Psi_{i,j}(0,0)B}{
\min\{h_1,h_2	\}(pq\rho)^2
}\right)+O(E_{i,j}),
\end{split}
\end{equation}
where
$$
E_{i,j}
=
\sum_{
\substack{-pq\rho/2<m,n\leq pq\rho/2 \\
(m,n)\neq (0,0)} }
\hspace{-0.5cm}
\min\left\{\frac{B}{h_1}, \frac{pq\rho}{|m|}\right\}
\min\left\{\frac{B}{h_2}, \frac{pq\rho}{|n|}\right\}
\frac{|\Psi_{i,j}(m,n)|}{
(pq\rho)^2}.
$$

We now come to the estimation of $\Psi_{i,j}(m,n)$, writing 
 $\rho=uvw^2$, 
with $u,v,w$ as in \eqref{eq:reading}. In particular $\gcd(pq,\rho)=1$.
It will  be convenient to put 
\begin{equation}\label{eq:la-rho}
\lambda(\rho)=u^{1/2}vw^2\gcd(w,h_1).
\end{equation}
Drawing together Lemmas \ref{lem:sigma-final} and \ref{lem:rho}
in Lemma \ref{lem:fact}, we deduce that 
$$
\Psi_{i,j}(m,n)\ll A^{\omega(\rho)} pq \lambda(\rho)   \gcd(pq,m,n)\gcd(u,\Delta(m,n))^{1/2},
$$
if $p\neq q$ and 
$$
\Psi_{i,j}(m,n)
\ll A^{\omega(\rho)} \1_{p\mid (m,n)} p^3 \lambda(\rho)   \gcd(p,m',n')\gcd(u,\Delta(m',n'))^{1/2}, 
$$
if $p=q$, where  
$(m',n')=(m,n)/p$ and 
$$
\1_{p\mid (m,n)}=\begin{cases}
1, & \mbox{if $p\mid (m,n)$,}\\
0, &\mbox{otherwise}.
\end{cases}
$$
If $p=q$ then 
$
\1_{p\mid (m,n)} p^3 \gcd(p,m',n') \leq  pq
\gcd(pq,m,n).
$
Recall that $A^{\omega(n)}=O_{A,\ve}(n^\ve)$ for any $\ve>0$.
In particular we have  $A^{\omega(\rho)}=O_\ve(B^\ve)$ in these estimates, since $\rho\leq 2B$.  
We therefore conclude that 
\begin{equation}\label{eq:pqp}
\Psi_{i,j}(m,n)
\ll_\ve B^\ve 
pq \lambda(\rho)   \gcd(pq,m,n) \gcd(u,\Delta(m,n))^{1/2},
\end{equation}
for any $p,q\in \cP$. In particular, taking $m=n=0$, it follows that 
\begin{equation}\label{eq:pqp-cor}
\Psi_{i,j}(0,0)
\ll_\ve B^\ve 
(pq)^2 \rho  \gcd(w,h_1).
\end{equation}

The following result will be useful when it comes to summing \eqref{eq:pqp} over the relevant $\rho$.

\begin{lemma}\label{lem:sum-rho}
Let $\Delta\in \NN$,  $\ve>0$ and let $\delta\in \{0,1\}$.
We have 
$$
\sum_{\rho\leq 2B/(h_1h_2)} \frac{\lambda(\rho) \gcd(u,\Delta)}{\rho^\delta}
\ll_\ve (\Delta B)^\ve \left(\frac{B}{h_1h_2}\right)^{3/2-\delta},
$$
where $\lambda(\rho)$ is given by \eqref{eq:la-rho} and 
$\rho=uvw^2$,
with $u,v,w$ as in \eqref{eq:reading}.
\end{lemma}

\begin{proof}
Let $S_\delta$ denote the sum in the lemma, for $\delta\in \{0,1\}$.
It suffices to handle the case  $\delta=1$,
since $S_0\ll (h_1h_2)^{-1}BS_1$. 
We will make use of the fact that 
$$
\sum_{n\leq N}\gcd(n,\Delta)= \sum_{d\mid \Delta} d \#\{n\leq N: d\mid n\}
\ll \tau(\Delta)N,
$$
where $\tau$ is the divisor function, together with  results that follow from it using partial summation. 
Recalling 
\eqref{eq:la-rho}, we 
note that 
$$
\frac{\lambda(\rho)}{\rho}=
\frac{\gcd(w,h_1)}{u_1^{1/2}}.
$$
Hence we have 
\begin{align*}
S_1&\leq
 \sum_{w\leq 2B/(h_1h_2)} 
\gcd(w,h_1)
\sum_{v\mid w}
\sum_{u\leq 2B/(h_1h_2v w^2)}
\frac{\gcd(u,\Delta)^{1/2}}{u_1^{1/2}}\\
&\ll\tau(\Delta)
\left(\frac{B}{h_1h_2}\right)^{1/2}
 \sum_{w\leq 2B/(h_1h_2)} 
\frac{\gcd(w,h_1)}{w}
\sum_{v\mid w}
\frac{1}{v^{1/2}}\\
&\ll_\ve \frac{\tau(\Delta) B^{1/2+\ve}}{(h_1h_2)^{1/2}}.
\end{align*}
We  complete the proof of the lemma by taking $\tau(\Delta)=O_\ve(\Delta^\ve)$.
\end{proof}

We now turn to an upper bound for 
$N_{1}(B,C;H)$, following \eqref{eq:N1-1} and \eqref{eq:N1-2}.
We start by analysing the main term $M_{i,j} $ in \eqref{eq:Mij}.
Suppose that $p\neq q$.
Then it follows from Lemmas~\ref{lem:fact} and \ref{lem:sigma-final} that 
\begin{align*}
M_{i,j}
&=
\frac{4\Sigma_i(p;0,0)\Sigma_j(q;0,0)\Phi(\rho;0,0)B^2}{h_1h_2(pq\rho)^2} \\
&=
\frac{4\max\{1,i\}\max\{1,j\}(1+O(p^{-1}+q^{-1}))\Phi(\rho;0,0)B^2}{h_1h_2\rho^2}\\
&=
\frac{4\max\{1,i\}\max\{1,j\}\Phi(\rho;0,0)B^2}{h_1h_2\rho^2} +O\left(
\frac{\Phi(\rho;0,0)B^2}{h_1h_2 \min\{p,q\}\rho^2}
 \right)
\end{align*}
for $i,j\in \{0,1,2\}$.
Recalling \eqref{eq:c-3}, we 
deduce that 
$$
\sum_{i,j\in \{0,1,2\}} 
\hspace{-0.3cm}
c_{i,j}(\alpha) M_{i,j}
=
\frac{4\Phi(\rho;0,0)B^2}{h_1h_2\rho^2}\left(\alpha-1\right)^2
 +O\left(
\frac{\Phi(\rho;0,0)B^2}{h_1h_2 \min\{p,q\}\rho^2}
 \right),
 $$
 if $p\neq q$.
Taking $\alpha=1$ therefore eliminates the main term in this expression.
 Suppose next that   $p=q\in \cP$. Then, returning to  \eqref{eq:Mij},  
we deduce from Lemma \ref{lem:fact}  that 
\begin{align*}
M_{i,j}
=
\frac{4\Sigma_{i+j}(p;0,0)\Phi(\rho;0,0)B^2}{h_1h_2p^2\rho^2} 
&\ll 
\frac{\Phi(\rho;0,0)B^2}{h_1h_2\rho^2},
\end{align*}
for $i,j\in \{0,1,2\}$.
It now follows that 
\begin{align*}
\sum_{h_1,h_2\leq H} \frac{\log^2 Q}{Q^2} &
\sum_{\rho\leq 2B/(h_1h_2)}\sum_{p,q\in \cP}
\left|
\sum_{i,j\in \{0,1,2\}} c_{i,j}(1) M_{i,j}
\right| \\
&\ll B^2\log^2 H
\frac{\log^2 Q}{Q^2} 
\sum_{\rho\leq 2B}
\frac{\Phi(\rho;0,0)}{\rho^2}\Upsilon,
\end{align*}
where
$
\Upsilon=
\sum_{p\neq q\in \cP}
\min\{p,q\}^{-1}+ \sum_{p\in \cP}
1\ll Q.
$
Next, 
Lemma \ref{lem:rho} implies that  $\Phi(\rho;0,0)\ll_\ve \rho^{1+\ve}\gcd(w,h_1)$,
whence 
$$
\sum_{\rho\leq 2B}
\frac{\Phi(\rho;0,0)}{\rho^2}
\ll_\ve B^\ve
\sum_{\rho\leq 2B} \frac{\gcd(w,h_2)}{\rho}.
$$
But
\begin{equation}\label{eq:fun}
\begin{split}
\sum_{\rho\leq 2B}
\frac{\gcd(w,h_1)}{\rho}
&\ll 
\sum_{w\leq 2B} 
\frac{\gcd(w,h_1)}{w^2}
\sum_{v\mid w}\frac{1}{v}
\sum_{u\leq 2B/(v w^2)}\frac{1}{u}\\
&\ll B^{2\ve},
\end{split}
\end{equation}
on recalling the decomposition $\rho=uvw^2$ from \eqref{eq:reading} and employing the bound 
$\sum_{1\leq n\leq N} 1/n\ll \log N\ll_\ve N^\ve$.
We conclude that 
\begin{equation}\label{eq:nye}
\sum_{h_1,h_2\leq H} \frac{\log^2 Q}{Q^2} \sum_{p,q\in \cP}
\sum_{\rho\leq 2B/(h_1h_2)}
\left|
\sum_{i,j\in \{0,1,2\}} c_{i,j}(1) M_{i,j}
\right| \ll_\ve \frac{B^{2+\ve}}{Q},
\end{equation}
on redefining the choice of $\ve$.

We now turn to the error terms in \eqref{eq:Sij'}. 
Firstly, it follows from \eqref{eq:pqp-cor} that 
$$
\frac{\Psi_{i,j}(0,0)B}{
\min\{h_1,h_2	\}(pq\rho)^2
}\ll_\ve
\frac{\gcd(w,h_1)B^{1+\ve}}{
\min\{h_1,h_2	\} \rho}.
$$
Next, 
we note from \eqref{eq:Delta} that $\Delta(m,n)=0$ if and only if $m=n=0.$ 
In view of \eqref{eq:pqp} we see that the contribution to the sum $E_{i,j}$ in \eqref{eq:Sij'} from $m=0$,
in which case $\Delta(0,n)=h_1^2n^2$, 
is 
\begin{align*}
&\leq \frac{B}{h_1}\sum_{
\substack{-pq\rho/2<n\leq pq\rho/2 \\
n\neq 0} }
\frac{pq\rho}{|n|}
\frac{|\Psi_{i,j}(0,n)|}{
(pq\rho)^2} \\
&\ll_\ve \frac{B^{1+\ve}}{h_1}
\frac{\lambda(\rho)}{\rho} 
\sum_{
\substack{-pq\rho/2<n\leq pq\rho/2 \\
n\neq 0} }
\frac{\gcd(pq,n)\gcd(u,h_1^2n^2)^{1/2}}{|n|}.
\end{align*}
Similarly, the contribution from  $n=0$ can be bounded by the same quantity, in which $h_1$ is replaced  by $h_2$.
We therefore conclude that terms with $mn=0$ and $(m,n)\neq (0,0)$ give an overall contribution
$$
\ll_\ve 
 \frac{\rho^{-1}\lambda(\rho) B^{1+\ve}}{\min\{h_1,h_2\}}
 \sum_{\substack{-pq\rho/2<k\leq pq\rho/2 \\
k\neq 0}} 
\frac{ \gcd(u,h_1h_2k)}{|k|}
$$
in \eqref{eq:Sij'}.

Next we consider the contribution to $E_{i,j}$  in \eqref{eq:Sij'} from $mn\neq 0$. 
Applying \eqref{eq:pqp} we see that this is
\begin{align*}
&\leq \sum_{
\substack{-pq\rho/2<m,n\leq pq\rho/2 \\
mn\neq 0} }
\frac{|\Psi_{i,j}(m,n)|}{|mn|}\\
&\ll_\ve B^\ve 
 pq \lambda(\rho) 
\sum_{
\substack{-pq\rho/2<m,n\leq pq\rho/2 \\
mn\neq 0} }
\frac{ \gcd(pq,m,n) \gcd(u,\Delta(m,n))^{1/2}}{|mn|}\\
&\ll_\ve B^\ve 
 pq \lambda(\rho) 
\sum_{
\substack{-pq\rho/2<m,n\leq pq\rho/2 \\
mn\neq 0} }
\frac{ \gcd(u,\Delta(m, n))
}{|mn|},
\end{align*}
since $\gcd(pq,u)=1$.

Combining this with our  estimates so far we conclude that 
\begin{align*}
|S_{i,j}-M_{i,j}|\ll_\ve~&
\frac{\gcd(w,h_1)B^{1+\ve}}{
\min\{h_1,h_2	\} \rho}\\
&+
\frac{  
\rho^{-1} \lambda(\rho)B^{1+\ve} }{\min\{h_1,h_2\}}
 \sum_{\substack{-pq\rho/2<k\leq pq\rho/2 \\
k\neq 0}} 
\frac{ \gcd(u,h_1h_2k)}{|k|}\\
&+pq \lambda(\rho)B^{\ve}
\sum_{
\substack{-pq\rho/2<m,n\leq pq\rho/2 \\
mn\neq 0} }
\frac{ \gcd(u,\Delta( m,n))}{|mn|}.
\end{align*}
We would now like to introduce a summation over $\rho\leq 2B/(h_1h_2)$.
For the first term we use \eqref{eq:fun}.
For the remaining two terms  we apply Lemma \ref{lem:sum-rho}.
This leads to the conclusion that 
\begin{align*}
\sum_{\rho\leq 2B/(h_1h_2)}
|S_{i,j}-M_{i,j}|
\ll_\ve~&
B^{5\ve}
\Big\{
\frac{B}{
\min\{h_1,h_2	\}}
+
 \frac{B^{3/2}}{\min\{h_1,h_2\}(h_1h_2)^{1/2}}\\
&+ \frac{(pq)^{1+\ve} B^{3/2}}{(h_1h_2)^{3/2}}
\Big\}.
\end{align*}
Now 
$$
\sum_{h_1,h_2\leq H} \frac{1}{\min\{h_1,h_2\}} \ll H \log H
$$
and 
$$
\sum_{h_1,h_2\leq H} \frac{1}{\min\{h_1,h_2\}(h_1h_2)^{1/2}} \ll H^{1/2}, \quad
\sum_{h_1,h_2\leq H}
 \frac{1}{(h_1h_2)^{3/2}} \ll 1.
$$
Using these estimates it follows that 
\begin{align*}
\sum_{h_1,h_2\leq H} \frac{\log^2 Q}{Q^2} \sum_{p, q\in \cP}
&\sum_{i,j\in \{0,1,2\}}
\sum_{\rho\leq 2B/(h_1h_2)}	
\left|
S_{i,j}-M_{i,j}\right| \\
&\ll_\ve B^{6\ve}\left\{HB+(H^{1/2}+
Q^2)B^{3/2}\right\}.
\end{align*}
By assumption $H\leq 2B$. Hence $HB\ll H^{1/2}B^{3/2}$.
Returning to \eqref{eq:N1-1} and \eqref{eq:N1-2}, and 
recalling \eqref{eq:nye}, 
 we now conclude that 
$$
N_{1}(B;H)
\ll_\ve
\frac{B^{2+\ve}}{Q}+
(H^{1/2}+
Q^2)B^{3/2+6\ve}.
$$
Taking 
$Q=B^{1/6}$, we  conclude the proof of Lemma \ref{lem:small-h} on redefining the choice of $\ve$.


\begin{thebibliography}{99}

\bibitem{BP}
E. Bombieri and J. Pila, 
The number of integral points on arcs and ovals. {\em Duke Math. J.} {\bf 59} (1989), 337--357. 

\bibitem{me}
T.D. Browning,
Equal sums of like polynomials.
{\em Bull. London Math. Soc.} {\bf 37} (2005), 801--808.


\bibitem{smoothI}
T.D. Browning and 
D.R. Heath-Brown, 
 The density of rational points on non-singular hypersurfaces, I.
{\em Bull. London Math. Soc.} {\bf 38} (2006), 401--410. 

\bibitem{deligne}
P. Deligne, La conjecture de Weil, I. {\em 
 Inst. Hautes \'Etudes Sci. Publ. Math.} {\bf 48} (1974), 273--307.

\bibitem{greaves}
G.R.H. Greaves, 
Representation of a number by the sum of two fourth powers. (Russian) {\em Mat. Zametki} {\bf 55} (1994), 47--58.



\bibitem{square}
D.R. Heath-Brown, The square sieve and consecutive square-free numbers.
{\em Math. Annalen} {\bf 266} (1984), 251--259

\bibitem{hooley-cubes}
C. Hooley, 
On  the numbers that are representable as the sum of two cubes.
{\em J. Reine Angew. Math.} {\bf 314} (1980), 146--173.


\bibitem{hoo1}
C. Hooley, 
On another sieve method and the numbers that are a sum of two $h$th powers.
{\em Proc. London Math. Soc.} {\bf 43} (1981),  73--109.

\bibitem{hooleysum}
C. Hooley, On exponential sums and certain of their 
applications.  {\em Journ\'ees Arithm\'etiques, 1980},  
London Math. Soc. Lecture Note Ser. {\bf 56}. (Cambridge Univ. Press, 
Cambridge-New York, 1982), 92--122.


\bibitem{hoo2}
C. Hooley, 
On another sieve method and the numbers that are a sum of two $h$th powers, II. 
{\em J. Reine Angew. Math.} {\bf 475} (1996), 55--75.


\bibitem{hux}
M.N.
Huxley, 
A note on polynomial congruences. {\em 
Recent progress in analytic number theory, Vol. 1 (Durham, 1979)},
193--196 (1981), 
Academic Press. 

\bibitem{munshi}
R. Munshi,
Density of rational points on cyclic covers of $\mathbb{P}^n$.
{\em J. Th\'eorie Nombres Bordeaux}
{\bf 21} (2009), 335--341.


\bibitem{weil}
A. Weil, Sur les courbes alg\'ebriques et les vari\'et\'es que s'en d\'eduisent.
{\em Actualit\'es Sci. Indust.} {\bf 1041} (1948). 


\bibitem{wooley} T.D. Wooley, Sums and differences of two cubic
  polynomials. {\em Monatsh. Math.} {\bf 129} (2000), 159--169. 


\end{thebibliography}
\end{document}